\documentclass{amsart}

\usepackage[english]{babel}
\usepackage[margin=95pt]{geometry}

\usepackage{amsmath, amssymb, amsthm}
\usepackage{enumerate}
\usepackage{setspace}

\newtheorem{theorem}{Theorem}

\newtheorem{lemma}[theorem]{Lemma}

\newtheorem*{theorem-non}{Theorem}

\newcommand{\RR}{\mathbb{R}}
\newcommand{\ZZ}{\mathbb{Z}}
\newcommand{\CC}{\mathbb{C}}
\newcommand{\QQ}{\mathbb{Q}}

\newcommand{\MM}{\mathfrak{M}}

\begin{document}

\title{Effective Methods for Norm-Form Equations}
\author[Prajeet Bajpai]{Prajeet Bajpai}
\address{Department of Mathematics, University of British Columbia, Vancouver, B.C., V6T 1Z2 Canada}
\email{prajeet@math.ubc.ca}

\date{\today}


\begin{abstract}
While effective resolution of Thue equations has been well understood since the work of Baker in the 1960s, similar results for norm-form equations in more than two variables have proven difficult to achieve. In 1983, Vojta was able to address the case of three variables over totally complex and Galois number fields. In this paper, we extend his results to effectively resolve several new classes of norm-form equations. In particular, we completely and effectively settle the question of norm-form equations over totally complex Galois sextic fields.
\end{abstract}

\maketitle

\section{Introduction}

Let $\alpha_1,\ldots, \alpha_k$ be $\QQ$-linearly independent elements of a number field $K$ and let $m$ be any fixed nonzero rational integer. The equation
\begin{equation}\label{normform}
N_{K/\QQ}(x_1\alpha_1 + \ldots + x_k\alpha_k) = m,
\end{equation}
where $x_1,\ldots,x_k$ are rational integers is  a \emph{norm-form equation}. The case $k=2$ corresponds to Thue equations, which Thue \cite{thue} showed in 1909 have at most finitely many solutions, provided $[K : \mathbb{Q}] \geq 3$. His proof was ineffective, meaning it could not be used-- even in principle-- to determine a complete list of solutions to any given equation. Effective solutions to Thue equations were first given via Baker's theorem on linear forms in logarithms. There is now an extensive literature on effective and explicit resolution for Thue equations, for some examples see \cite{biluhanrot}, \cite{thomas}, \cite{togbevoutierwalsh}, \cite{tdw1}, \cite{tdw2}.

Concerning the more general case of equation \eqref{normform} for $k\ge 3$, Schmidt \cite{schmidt1971} proved in 1971 that such equations have only finitely many solutions provided they satisfy a certain `non-degeneracy condition'. In a subsequent paper, he also showed in the `degenerate' case that the solutions lie in finitely many families \cite{schmidt1972}. For a thorough account of his characterisation of solutions to norm-form equations, see \cite{schmidtbook}. These results rely on his Subspace Theorem, meaning again that they are ineffective. Certain specific families of norm-form equations have since been solved effectively, via a wide variety of techniques including  Pad\'e approximation \cite{bennettnormform},  Skolem's $p$-adic method \cite{stroekertzanakis}, and, most frequently, through bounds for linear forms in logarithms. In 1970, Gy\H  ory and Lov\' asz \cite{gyorylovasz} were able to (effectively) solve \emph{non-degenerate} norm-form equations in three variables over CM fields by reducing them to Thue equations. Subsequently, Gy\H ory \cite{gyory1981} and Bugeaud-Gy\H ory \cite{bugeaudgyory} resolved certain types of norm-form equations by reducing them to Thue equations over relative extensions of number fields. However, their hypotheses are quite restrictive, requiring 
$$
[K:\QQ(\alpha_1,\ldots,\alpha_{n-1})] \ge 3
$$
 and only considering solutions with $x_n \neq 0$. In particular, this often rules out effective solution of equations corresponding to a ``power basis'',  e.g. $N(x+y\theta + z\theta^2) = 1$, which are commonly considered examples. The question of effective resolution of general norm-form equations remains open in most cases.

In his PhD thesis in 1983, Vojta \cite{vojta} used a pigeonhole argument to demonstrate an effective solution for norm-form equations in three variables over totally complex Galois number fields -- with no restriction on degree. The purpose of this paper is to  extend these results in several directions. First, we prove that in Vojta's theorem one may drop the requirement that $K$ be Galois. This gives us
\begin{theorem}\label{nongal}
Let $K$ be a totally complex (not necessarily Galois) number field and let $\alpha_1,\alpha_2,\alpha_3$ be $\QQ$-linearly independent elements of $K$ such that the ratios $\alpha_j/\alpha_k$ generate $K$ over $\QQ$. Then, for any fixed integer $m$, the solutions of the norm-form equation
\[
N_{K/\QQ}(x_1\alpha_1 + x_2\alpha_2 + x_3\alpha_3) = m
\]
can be effectively determined.
\end{theorem}
Next, we give a complete and effective resolution to the question of norm-form equations over totally complex Galois number fields of degree $6$. In this regard, we prove
\begin{theorem}\label{sexticthm}
Let $K$ be a totally complex Galois field with $n = [K:\QQ] = 6$ and let $\alpha_1,\ldots,\alpha_5$ be $\QQ$-linearly independent elements of $K$ such that the ratios $\alpha_j/\alpha_k$ generate $K$ over $\QQ$. Then, for any fixed integer $m$, the solutions of the norm-form equation
\begin{equation}\label{sexticform}
N_{K/\QQ}(x_1\alpha_1 + x_2\alpha_2 + x_3\alpha_3 + x_4\alpha_4 + x_5\alpha_5) = m
\end{equation}
can be effectively determined.
\end{theorem}
In both these theorems, by ``effectively determined'' we mean that the solution set is the union of
\begin{enumerate}[(i)]
\item finitely many families of solutions, which can each be explicitly described by a combination of linear and congruence conditions on the exponents of two chosen fundamental units, and
\item finitely many solutions outside the above families, whose height can be bounded by an effective constant.
\end{enumerate}
See Section \ref{families} for a discussion of how these families relate to the description of \cite{schmidt1972}. The LMFDB \cite{lmfdb} provides a ready source of examples of totally complex Galois sextic fields-- at the time of writing, there are 114196 such fields in the database.

Finally, we are able to extend these methods to handle four variables over higher degree fields in the case of some special norm form equations -- namely those in a so-called `power basis'. We have
\begin{theorem}\label{generalthm}
Let $K$ be a totally complex Galois field with $[K:\QQ] \ge 6$ and let $\alpha$ be a primitive element of $K$, i.e. we have $K = \QQ(\alpha)$. Let $m$ be any fixed rational  integer. The solutions to the norm-form equation
\begin{equation*}
N_{K/\QQ} (x_1 + x_2\alpha + x_3\alpha^2 + x_4\alpha^3) = m
\end{equation*}
can be effectively determined.
\end{theorem}
In the case of Theorem \ref{generalthm}, any (infinite) families of solutions can again be given by sets of linear conditions on the exponents, along with a possible congruence condition arising from the torsion in the unit group of the number field. ``Effectively determined'' means, as before, that these families can be explicitly described, and the heights of solutions outside these families can be bounded by an effective constant. The requirement that $K$ be Galois and the ratios $\alpha_i/\alpha_j$ generate $K/\QQ$ ensures that the resulting norm form is irreducible. The proofs of Theorems \ref{sexticthm} and \ref{generalthm} both utilise the technique of ``matching'' introduced in \cite{bajben}.

\section{Preliminaries}\label{prelim}

Let $K$ be a Galois number field, set $n = [K:\QQ]$ and let $\alpha_1,\ldots,\alpha_k\; (k<n)$ be $\QQ$-linearly independent elements of $K$ such that the ratios $\alpha_i/\alpha_j$ generate $K$ over $\QQ$. Let $\mathrm{Gal}(K/\QQ) = \{1 = \sigma_1,\ldots,\sigma_n\}$ be the Galois group of $K$ over $\QQ$ and define $B$ to be the matrix
\[
\begin{pmatrix}
\sigma_1(\alpha_1) & \sigma_2(\alpha_1) & \cdots & \sigma_n(\alpha_1) \\
\sigma_1(\alpha_2) & \sigma_2(\alpha_2) & \cdots & \sigma_n(\alpha_2) \\ 
\vdots & \vdots & \ddots & \vdots \\ 
\sigma_1(\alpha_k) & \sigma_2(\alpha_k) & \cdots & \sigma_n(\alpha_k) 
\end{pmatrix}
\]
so that the condition $\mathrm{disc}(K) \neq 0$ implies that $B$ has rank $k$. The columns of $B$ therefore satisfy $n-k$ linear relations. Let $A$ be an $(n-k)\times n$ matrix (of rank $n-k$) where each row gives one such relation. Setting
\[
\mathcal{L}^{\sigma} = x_1\sigma(\alpha_1) + \cdots + x_k\sigma(\alpha_k)
\] 
for any $\sigma \in \mathrm{Gal}(K/\QQ)$, we see that the matrix $A=(a_{ij})$ leads to a system of $n-k$ equations among the linear forms $\mathcal{L}^\sigma$. Specifically, we obtain equations of the shape
\[
a_{i1} \mathcal{L}^{\sigma_1} + a_{i2} \mathcal{L}^{\sigma_2} + \cdots + a_{in} \mathcal{L}^{\sigma_n} =0 \, , \quad 1 \le i \le n-k.
\]
Now consider the norm-form equation
\begin{equation}\label{gennorm}
N_{K/\QQ} (x_1\alpha_1 + \cdots + x_k\alpha_k) = m
\end{equation}
where $m$ is some fixed integer. Letting $d$ be the smallest positive integer such that $d\alpha_j$ is in the ring of integers $\mathcal{O}_K$ of $K$ for each $1\le j \le k$, we see that solutions to the above equation are equivalent to the solutions to
\[
N_{K/\QQ} (x_1\cdot d\alpha_1 + \cdots + x_k\cdot d\alpha_k) = md^n
\]
so we may as well restrict \eqref{gennorm} to the case where the $\alpha_j$ are all in $\mathcal{O}_K$. Then all the $\mathcal{L}^\sigma$ represent integral elements of $K$. Next we note that, up to units, there are only finitely many $\mu \in \mathcal O_{K}$ with $N_{K/\QQ}(\mu) = m$, and so we can write $x_1\alpha_1 + \cdots + x_k\alpha_k = \mu u$ with $\mu$ chosen from a finite (effectively bounded) set and $u$ a unit in $\mathcal{O}_K$. This allows us to replace the matrix $A$ with a finite collection of matrices
\[
A_\mu = \Big(a_{ij}\sigma_j(\mu) \Big),
\]
one for each choice of $\mu$, now giving equations to be solved in \emph{units} of $\mathcal{O}_K$. Thus we can reduce the question of solving \eqref{gennorm} to the question of solving certain systems of unit equations (for the proof of Theorem \ref{generalthm} we will use some statements about the rank of submatrices of $A$, and we will show there that these rank conditions continue to hold for $A_\mu$). Writing $\mathcal{L}^{\sigma_i} = u_i$, each of the $A_\mu$ gives a system of unit equations of the form
\begin{equation}\label{basicunit}
a_{i1} u_1 + a_{i2} u_2 + \cdots + a_{in} u_n = 0\, , \quad 1 \le i \le n-k.
\end{equation}

If $k=2$ (the case of Thue equations), then we obtain a system of $n-2$ equations in $n$ variables. After row-reduction each equation is left with only three non-zero terms terms. Three term unit equations can all be effectively solved following Baker's method, and so, as is well-known, the effective resolution of Thue equations is a straightforward matter.

If $k=3$, then each equation has four terms after row reduction. In this case, if $K$ is assumed to be totally complex then Vojta \cite{vojta} was able to show via a pigeonhole argument that again the solutions may be effectively determined. The key point is to use Baker's theorem to show that three terms in each equation are necessarily ``large'' at every place (i.e. comparable to the largest term), and so if the system is suitably generic then in fact one term ends up being large at all places. This of course contradicts the product formula for places of a number field -- or, more simply, the fact that the product of all conjugates was assumed to equal $\pm 1$, since the terms are units. The ``not-generic'' case is dealt with by reducing to a smaller system of equations that is usually simpler.

The ``matching'' procedure, described below, allows us to reduce the number of terms in the unit equations that arise from a given norm-form -- at the cost of a small increase in height of the coefficients. For example, in the sextic case we go immediately from a six-term equation to a three-term equation, and the increase is small enough that the new equation can still be solved by Baker's method. A more generalized matching procedure then allows us to address the situation in Theorem \ref{generalthm}. The essential workings of the proof are the same as in the sextic case.

Before proceeding, we state a result on lower bounds for linear forms in complex logarithms that we will need. For $\alpha\in K$ with $[K : \mathbb{Q}]=n$,  we define the absolute logarithmic height of $\alpha$ as
\[
h(\alpha) = \frac{1}{n}\left( \log|a| + \sum_{i=1}^{n}\log\max\{1,|\alpha^{(i)}|\}  \right),
\]
where $a$ is the leading coefficient of the minimal polynomial of $\alpha$ over $\ZZ$, and the $\alpha^{(i)}$ are the conjugates of $\alpha$. For linear forms in an arbitrary number of complex logarithms, the following result of Matveev \cite{matveev} is essentially the best known:

\begin{theorem}[Matveev, 2000]\label{matveev}
Let $\alpha_1,\ldots,\alpha_m$ be non-zero elements of $K$, where $[K : \mathbb{Q}]=n$, and $b_1,\ldots b_m$ be integers such that
\[
\Lambda = |b_1\log \alpha_1 + \cdots + b_m\log \alpha_m| \neq 0.
\]
If $K\subseteq\RR$ put $\varkappa=1$, else put $\varkappa=2$. Let $B = \max\{|b_1|,\ldots,|b_m|\}$ and 
\[
A_j \ge \max\{nh(\alpha_j),|\log\alpha_j|,0.16)\},\quad 1\le j\le m .
\]
Then we have
\[
\log|\Lambda| \ge -\frac1\varkappa\left( em \right)^\varkappa 30^{m+3}m^{3.5}n^2\log(en)\log(eB)A_1\cdots A_m.
\]
\end{theorem}

Finally, for a solution $\overline u = (u_1,u_2,\ldots,u_n)$ to a unit equation such as \eqref{basicunit} define
\[
H(\overline u) = \prod_\nu \max \{ \lVert u_1 \rVert_\nu , \lVert u_2 \rVert_\nu, \ldots, \lVert u_n \rVert_\nu \}
\]
and $h(u) = \log H(u)$.

\subsection{Proof of Theorem \ref{nongal}}
In this section, we give a simiplified proof of Vojta's theorem from \cite{vojta} that does not require the field $K$ to be Galois over $\QQ$. This will give us Theorem \ref{nongal}. So let $K, \alpha_1,\alpha_2,\alpha_3$ be as in the statement of Theorem 1 and consider
\[
N_{K/\QQ}(x_1\alpha_1 + x_2\alpha_2 + x_3\alpha_3) = m.
\]
Immediately we restrict our attention to the case $\alpha_1,\alpha_2,\alpha_3 \in \mathcal O_K$ since this can be ensured by suitably altering $m$, as described above. Next we restrict to the case $m = \pm 1$, from which the general case will follow. Under these restrictions $u = x\alpha_1 + y\alpha_2 + z\alpha_3$ is a unit in $K$. Let $L/K$ be a normal closure of $K$, $\sigma_1, \ldots, \sigma_n$ be the $n$ distinct embeddings of $K$ into $L$ and set $u_i = \sigma_i(u)$. Set $B$ to be the matrix
\[
B = \begin{pmatrix}\sigma_1(\alpha_1) & \sigma_2(\alpha_1) & \cdots & \sigma_n(\alpha_1) \\
\sigma_1(\alpha_2) & \sigma_2(\alpha_2) & \cdots & \sigma_n(\alpha_2) \\
\sigma_1(\alpha_3) & \sigma_2(\alpha_3) & \cdots & \sigma_n(\alpha_3) 
\end{pmatrix}
\]
so by the non-vanishing of the discriminant of $K$ we know that $B$ has rank $3$. Let $A = (a_{ij})$ be an $(n-3)\times n$ matrix of linear relations satisfied by the columns of $B$ and note this implies we have the $n-3$ linear equations
\[
a_{i1} u_1 + a_{i2} u_2 + \cdots + a_{in}u_n \, , \qquad 1 \le i \le n-3
\]
in units of $L$. Let $\overline u$ denote such a potential solution $(u_1,\ldots, u_n)$. We begin by stating Lemma 2.19 of \cite{vojta} which is the following
\begin{lemma}[Vojta] \label{lemvoj}
Let $A$ and $B$ be $r\times n$ and $s\times n$  matrices, respectively, with $r+s = n$. Assume $A$ and $B$ have rank $r$ and $s$, respectively, and that rows of $A$ are orthogonal to rows of $B$. Finally, assume the last $s$ columns of $B$ have rank $s$. Then the first $r$ columns of $A$ have rank $r$.
\end{lemma}

Recall, following \cite{vojta}, how this lemma implies that any $n-2$ columns of $A$ have rank $n-2$. Since the ratios $\alpha_i/\alpha_j$ generate $K$ over $\QQ$, any two columns of $B$ are linearly independent. Now choose $n-2$ columns of $A$, and let indices $k,\ell$ correspond to the columns not chosen. Then columns $k,\ell$ of $B$ are linearly independent, so there exists an index $m$ such that columns $k,\ell, m$ of $B$ give a submatrix of rank 3. We can reorder simultaneously the columns of $A$ and $B$ to put columns $k,\ell,m$ as the rightmost columns. Then the matrices $A$ and $B$ satsify the conditions of the above Lemma, so the first $n-3$ columns of (the reordering of) $A$ have rank $n-3$. To finish, we note that these $n-3$ columns were among the $n-2$ we chose initially.

Returning to our set of matrix equations, fix an infinite place $\nu$ of $L$ and an embedding $K\hookrightarrow L$. We may assume, after relabelling the $u_i$, that
\[
\lVert u_1 \rVert_{\nu} \ge \lVert u_2 \rVert_{\nu} \ge \cdots \ge \lVert u_n \rVert_{\nu}
\]
and, since $K$ is totally complex, that
\begin{equation}
\lVert u_1 \rVert_{\nu} = \lVert u_2 \rVert_{\nu},  \lVert u_3 \rVert_{\nu} = \lVert u_4 \rVert_{\nu}, \ldots \lVert u_{n-1} \rVert_{\nu} = \lVert u_n \rVert_{\nu}.
\end{equation}
After row reduction, our matrix $A$ has at most one non-zero entry $a_{ij}$ with $i < j < n-1$ since there is at most one `non-pivot' column among the first $n-2$. Let this be column $k$.  The first $k$ rows of $A$ look like
\[
\begin{pmatrix}
a_1 & &  & & b_1 & 0 & \cdots & 0  & c_1 & d_1 \\
 & a_2 & &  & b_2 & 0 & \cdots & 0  & c_2 & d_2\\
 &  & \ddots & & \vdots  & \vdots & \ddots  & \vdots  & \vdots   & \vdots \\
   &   & & a_k  & b_k  & 0 & \cdots & 0 & c_k   & d_k 
\end{pmatrix}
\]
while the last $n-3-k$ have zeroes in columns $\le k+1$ giving the following submatrix:
\[
\begin{pmatrix}
0&\cdots &0&  a_{k+1} & &  &  c_{k+1} & d_{k+1} \\
 \vdots&\ddots &\vdots&  &\ddots & & \vdots & \vdots \\
 0&\cdots &0&  & & a_{n-3}&  c_{n-3} & d_{n-3}
\end{pmatrix}
\]
with only three non-zero entries in each row.

If $b_1= 0$ then we have $c_1$ or $d_1$ non-zero, and so in particular that $| u_{n-1} |_{\nu} \gg | u_1 |_\nu$.  Next, if $a_1u_1 + b_1u_k = 0$ we have $|u_k|_\nu \gg |u_1|_\nu$ so we check $a_{i}u_i + b_{i}u_k$ for all $1<i<k$. If $b_i = 0$ in any of these cases again we get $| u_{n-1} |_{\nu} \gg | u_1 |_\nu$. If $a_{i}u_i + b_{i}u_k = 0$ for all $1<i<k$ then our system of equations splits into a collection of unit equations in two and three variables (for rows $i<k$ and rows $i>k$ respectively) which can all be solved effectively. So we may assume for some $1<i<k$ we have $a_{i}u_i + b_{i}u_k \neq 0$ and also $|u_k|_\nu \gg |u_1|_\nu$.

Then, using the bounds from Theorem \ref{matveev} we deduce
\[
\lVert u_{n-1} \rVert_{\nu} \ge c{\lVert u_k \rVert_\nu}h(\overline u)^{-d} \ge c'{\lVert u_1 \rVert_\nu}h(\overline u)^{-d}
\]
with effective constants $c,c',d$ depending on the $a_i$ and $K$ but not on the $u_i$. Now recall $\lVert u_{n-1}\rVert_\nu = \lVert u_n \rVert_\nu$ since $K$ is totally complex. Moreover we can suppose $\lVert u_n \rVert_\nu < 1$ since the product of the $u_i$ equals $1$ and all conjugates $u_i$ cannot lie on the unit circle unless they are roots of unity. Thus we get the inequality
\[
1 > \lVert u_n \rVert_{\nu} \ge c'{\lVert u_1 \rVert_\nu}h(\overline u)^{-d}
\]
which means $\lVert u_1\rVert_\nu$, and thus $h(\overline{u})$, is effectively bounded (note that $h(\overline u)$ is logarithmic in size compared to $u$).

To go from $m=\pm 1$ to arbitrary $m$, we recall as described above that we only have to consider finitely many collections of matrix equations $A_\mu$. Furthermore, these are all obtained from $A$ by multiplying each column by a different constant element of $L$. In particular, the condition that any $n-2$ columns of $A_\mu$ have rank $n-3$ remains true, since this condition holds for $A$.  Thus the same argument gives a bound on the heights of solutions also in the case of an arbitrary integer $m$.

\subsection{Matching Units}\label{matching}

The main tool in the remaining proofs is the ``matching'' procedure of \cite{bajben}. We reproduce the ideas here for completeness, with only slight modification so the procedure can be applied to more general systems of equations.

First, we prove a small lemma to show that three-term unit equations can be solved even if the coefficients in the equation are allowed to ``vary'' up to small height. The proof is essentially the same as the usual proof for three-term equations via Baker's method-- in particular see \cite{gyoryyu} for bounds in the usual case where the dependence on heights of (fixed) coefficients is made explicit.

\begin{lemma}\label{threetermlemma}
Let $u,v$ and $w$ be units in a number field $K$, write $\overline{u} = (u,v,w)$ and let $\alpha, \beta$ and $\gamma$ be non-zero elements of $K$ satisfying
\[
\max\{h(\alpha), h(\beta), h(\gamma)\} \le c + d \log h(\overline u)
\]
for some effective positive constants $c$ and $d$. Then the height of solutions to the unit equation
\[
\alpha u + \beta v + \gamma w = 0
\]
can be effectively bounded.
\end{lemma}
\begin{proof}
 First, we dehomogenize the equation to consider instead
\[
\alpha u + \beta v = -\gamma
\]
noting that any solution $(u_1,v_1)$ to the above equation gives a primitive solution $(u_1,v_1,1)$ to the homogeneous equation -- i.e. a solution $(u_1,v_1,w_1)$ up to a common multiplicative factor. Next, since we are interested in effective bounds, we may as well restrict to considering only $\overline{u}$ of sufficiently large height -- say $h(\overline u) \ge A \gg 1$ for some positive constant $A$ to be determined.

Then there must be at least one infinite place $\nu$ of $K$ such that 
\[
\max\{|u|_\nu, |v|_\nu\} \ge h(\overline{u})^{1/s}\ge A^{1/s}
\]
where $s$ is the number of infinite places of $K$. Assuming without loss of generality that $|u|_\nu \ge |v|_\nu$ for this place $\nu$, it follows that if $A$ is large enough then $|u|_\nu \gg 1$. We can further ensure
\[
|v|_\nu \ge |\alpha/\beta|_\nu |u|_\nu - |\gamma/\beta|_\nu \gg 1
\]
as well, assuming $A$ large enough, since the coefficients $\alpha,\beta$ and $\gamma$ are of small height relative to $(u,v)$.

Now we have
\[
| \gamma |_\nu = | \alpha u + \beta v |_\nu \ge |\beta v|_\nu \left| \frac{-\alpha u}{\beta v} - 1  \right|_\nu
\]
If $|\alpha u/\beta v |_\nu \le 0.9$ say, then
\[
|\gamma |_\nu \ge 0.1 | \beta v |_\nu \ge \frac{1}{9} | \alpha u |_\nu,
\]
whence $ c' |u|_v \log^{-d'}h(\overline u) \le 1$ for some suitable positive constants $c', d'$ depending upon $c,d$. This implies that $h(\overline u)$ is effectively bounded, since we had chosen $\nu$ such that $|u|_\nu \ge h(\overline u)^{1/s}$. If $|\alpha u/\beta v |_\nu > 0.9$ then
\[
\left| \frac{-\alpha u}{\beta v} - 1  \right|_\nu \ge \tfrac12 \log \left( \frac{-\alpha u}{\beta v} - 1  \right)
\]
where we choose the principal branch of the complex log taken relative to the embedding $K\hookrightarrow \CC$ corresponding to $\nu$. Let $\varepsilon_1,\ldots,\varepsilon_r$ be a system of fundamental units for $K$, $m$ the number of roots of unity in $K$, and $\zeta \in K$ a primitive $m$th root of unity. Writing $u/v = \zeta^k\varepsilon_1^{a_1}\cdots\varepsilon_r^{a_r}$, we see that we need to consider the linear form
\[
\Lambda =  a_1\log\varepsilon_1 + \cdots + a_r\log\varepsilon_r + a_{r+1} (2\pi i /w) + \log(-\alpha/\beta)
\]
where $a_{r+1}$ can be adjusted to account for the differences between $\log(\varepsilon_i^{a_i})$ and ${a_i}\log \varepsilon_i$. Applying Theorem \ref{matveev} and again assuming $A$ large enough gives 
\[
|\Lambda| \ge e^{-C\log(eB)A^\prime}
\]
where $B = \max\{a_1,\ldots,a_{r+1}\}$, $A^\prime = \max\{[K : \mathbb{Q}] h(-\alpha/\beta),|\log(-\alpha/\beta)|,0.16 \}$ and $C$ depends on $[K : \mathbb{Q}]$  and the heights of the generators of the group of $S$-units (see \cite{hajdu} for some examples of how these heights can be bounded). We have that $B\ll h(u_1/u_2)\ll h(\overline{u})$, and by our assumption on the heights of the $a_i$ we have $A_{n+2}\ll \log h(\overline{u})$. Putting this all together, it follows that
\[
c' |u|_\nu e^{-d'\log h(\overline u) - \log^2 h(\overline u)} \le 1
\]
for suitable constants $c',d'$. Thus again $h(\overline u)$ can be bounded by an effective constant.

\end{proof}

Now we return to our discussion of matching. Consider a solution $\overline u = (u_1,\ldots, u_n)$ to a given unit equation
\begin{equation}\label{matchdef}
a_1u_1 + a_2u_2 + \cdots + a_n u_n = 0
\end{equation}
with $a_i \in K, (1 \le i \le n)$.  We say two units $u_\ell, u_{k}$ in the above equation can be `matched' if we can write
\[
a_{\ell}u_{\ell} + a_{k}u_{k} = au
\]
with $u$ a unit and with $h(a) \le c + d\log h(\overline u)$ for some effective positive constants $c,d$ that may depend on the number field $K$ as well as the heights of the coefficients $a_i$, but are independent of $\overline{u}$. If we can reduce equation \eqref{matchdef} to a three-term equation by matching sufficiently many units, then by Lemma \ref{threetermlemma} we see that the height of solutions to \eqref{matchdef} can be effectively bounded.

Our main strategy for matching units is to show that they are of comparable size at every place, and deduce that they must then be essentially the \emph{same} unit up to a multiplicative factor of small height. So suppose there exist $1 \le \ell, k, \le n$ and units $u_\ell, u_k$ in the equation \eqref{matchdef} such that for some positive constants $c,d$ and \emph{all} infinite places $\nu$ of $K$ we have
\begin{equation}\label{samesize}
\lVert u_\ell/u_k \rVert_{\nu} \le c \cdot h(\overline{u})^d.
\end{equation}
Then writing
\[
a_\ell u_\ell + a_ku_k = \left(  a_\ell u_\ell/u_k + a_k  \right) u_k
\]
it follows from \eqref{samesize} that
\[
h \left(  a_\ell u_\ell/u_k + a_k  \right) \le  sd \log h(\overline u ) + s\log c +  \max\{h( a_\ell), h( a_k) \}\cdot\log 2
\]
where $s$ is the number of infinite places of $K$. This satisfies our definition of matching.

\section{Norm-Form Equations over Sextic Fields}

Let $K$ and the $\alpha_i  \, (1\le i\le 6)$ be as in the statement of Theorem \ref{sexticthm}. As discussed in Section \ref{prelim}, we may restrict to $m=\pm 1$ in equation \eqref{sexticform}, and assume the $\alpha_i$ lie in $\mathcal{O}_K$. Let $\mathrm{Gal}(K/\QQ) = \{\sigma_1,\ldots,\sigma_6\}$. Since the $\alpha_j$ are $\QQ$-linearly independent and their ratios generate $K$ over $\QQ$, the six vectors
\[
v_j = \big[\, \sigma_j(\alpha_1), \sigma_j(\alpha_2), \sigma_j(\alpha_3), \sigma_j(\alpha_4), \sigma_j(\alpha_5)\, \big] , \quad 1 \le j \le 6
\]
satisfy (up to multiplication by a non-zero constant) exactly one linear relation $a_1v_1 + \cdots + a_6v_6 = 0$. Setting
\[
u_j = \sigma_j(x_1\alpha_1 + x_2\alpha_2 + x_3\alpha_3 + x_4\alpha_4 + x_5\alpha_5)
\]
and noting that equation \eqref{sexticform} with $m=\pm 1$ implies that the $u_j$ are all units in $K$, we obtain a corresponding unit equation
\[
a_1u_1 + a_2u_2 + a_3u_3 + a_4u_4 + a_5u_5 + a_6u_6 = 0.
\]
Since $[K:\QQ]=6$ we have $|S_\infty| = 3$. Fix $\nu\in S_\infty$. After relabelling the $u_j$ if necessary, we may assume
\[
\lVert u_1 \rVert_\nu \ge \lVert u_2 \rVert_\nu \ge \lVert u_3 \rVert_\nu \ge \lVert u_4 \rVert_\nu \ge \lVert u_5 \rVert_\nu \ge \lVert u_6 \rVert_\nu
\]
and moreover that each pair
\[
(u_1,u_2),\; (u_3,u_4), \; (u_5,u_6)
\]
yields a pair of complex conjugates under the embedding corresponding to $\nu$. We will call a unit $u_j$ ``large'' at $\nu$ if
\begin{equation}\label{definelarge}
\lVert u_j \rVert_\nu \ge c \lVert u_1 \rVert_\nu h(u)^{-d}
\end{equation}
for some positive constants $c$ and $d$, independent of $h(u)$. Similarly, $u_j$ would be large at $\nu'$ if it satisfied the inequality  $\lVert u_j \rVert_\nu \ge c' \lVert u_k \rVert_\nu h(u)^{-d'}$ where $u_k$ is the conjugate with largest absolute value at $\nu'$.

Now if $a_1u_1 + a_2u_2 = 0$ then also $a_3u_3 + a_4u_4 + a_5u_5 + a_6u_6 = 0$ and so we have a pair of vanishing subsums. Both these equations can be effectively solved -- for the latter in particular we note it is a 4-term equation and since $|S_\infty|=3$ it is solvable by Vojta's methods \cite{vojta}. Otherwise we have $a_1u_1 + a_2u_2 \neq 0$ and so an application of Theorem \ref{matveev} guarantees that $a_1u_1 + a_2u_2$ is ``large'' at $\nu$. More precisely, it gives
\[
\lVert u_3 \rVert_\nu \ge c \cdot \lVert a_1u_1 + a_2u_2 \rVert_\nu \ge c_1 \lVert u_1 \rVert _\nu h(u)^{-d_1}
\]
where the constants $c_1, d_1$ are effective. Since $u_3$ and $u_4$ are complex conjugates at $\nu$ this immediately gives
\[
\lVert u_4 \rVert_\nu  = \lVert u_3 \rVert_\nu \ge c_1 \lVert u_1 \rVert _\nu h(u)^{-d_1}
\]
and so at most the two terms $u_5$ and $u_6$ can fail to be large at $\nu$ in the sense of \eqref{definelarge}. The action of $\mathrm{Gal}(K/\QQ)$ on the infinite places of $K$ means that sizes at $\nu$ determine the sizes at all remaining places. In particular, for each infinite place of $K$, it follows that at most two terms are potentially ``not-large'' at that place. The action of $\mathrm{Gal}(K/\QQ)$ is also transitive on the $u_j$, and complex conjugation at $\nu$ swaps $u_5$ and $u_6$, so in fact $u_5$ and $u_6$ cannot be among the two smallest units at any place besides $\nu$. This means they are necessarily \emph{large} at the remaining places, i.e. for $\nu' \neq \nu$ we get
\begin{align*}
\lVert u_{5} \rVert_{\nu'} \ge c_1 \cdot \max_{1\le j \le 6} \lVert u_j \rVert_{\nu'}\, h(\overline u)^{d_1} \\
\lVert u_{6} \rVert_{\nu'} \ge c_1 \cdot \max_{1\le j \le 6} \lVert u_j \rVert_{\nu'}\, h(\overline u)^{d_1}.
\end{align*}
In particular this means
\[
c_1\cdot h(\overline{u})^{-d_1} \le\;\, \lVert u_5/u_6 \rVert_{\nu'} \le c_1^{-1} \cdot h(\overline{u})^{d_1}.
\]
Of course at $\nu$ we directly have $\lVert u_5/u_6 \rVert = 1$ since the two units are complex conjugates at $\nu$. So $u_5$ and $u_6$ satisfy \eqref{samesize} and can be matched. By our arguments above, we in fact have something stronger -- for every $1\le j \le 6$ there is a $1\le k \le 6$, $j\neq k$, and a place $\nu_j$ such that $(u_j,u_k)$ are complex conjugates at $\nu_j$ and large at all places other than $\nu_j$. Thus \emph{every} unit can be matched with one other, and so we reduce to a three-term equation, with the cost of a potential logarithmic growth in the size of the coefficients. By Lemma \ref{threetermlemma}, this equation can be solved effectively and so the solutions to our norm-form equation can be effectively determined.

Thus we see that norm-form equations over totally complex sextic fields are always effectively solvable. If the norm-form is in four variables, the solution set is simply the intersection of the solutions to two five-variable norm forms extending the chosen one. If we have a norm-form equation in six variables, then in fact there are always infinitely many solutions, and these can be explicitly described via Dirichlet's Unit Theorem for Orders (see \cite{schmidtbook} for details, and \cite{akhtarivaaler} for some results on heights of solutions in this case). The case of three variables was of course already known due to Vojta \cite{vojta}, and in two variables we have the usual case of Thue equations.

\subsection{Infinite Families and Vanishing Subsums}\label{families}

In \cite{schmidt1971}, Schmidt proved that norm-form equations satisfying a certain non-degeneracy condition have only finitely many solutions. He extended his results in \cite{schmidt1972}, where he proved that the (possibly infinite) solutions in the degenerate case come in finitely many ``maximal families'' and gave an explicit description of how these families arise from subfields of $K$. We recount his classification, following closely the notation and presentation of \cite{schmidt1971} and \cite{schmidt1972}, and show how our results fit into this framework.

As the $x_1,\ldots,x_k$ range over the rational integers, the form $x_1\alpha_1+\cdots+x_k\alpha_k$ spans a $\ZZ$-module in $K$. Call this module $\mathfrak M$. In Theorem \ref{sexticthm}, we have considered the case $k=5$, while $K$ has degree $6$. Thus, in this case in particular, the module is not of \emph{full rank} in $K$, i.e. $\mathrm{rank}\, \mathfrak M < [K:\QQ]$. However, $\mathfrak M$ may contain a submodule $\mathfrak M '$ such that $\mathfrak M '$ is full in some proper subfield of $K$. In this case, the norm form equation $N(\overline x) = m$ may have infinitely many solutions with $\overline x$ in $\mathfrak M$. For example, $\mathfrak M'$ may contain the entire ring of integers of a subfield $L$ of $K$ (where $L$ is not $\QQ$ or imaginary quadratic, so that it has an infinite unit group) giving rise to infinitely many solutions to the equation $N(\overline x) = 1$ with $\overline x$ in $\mathfrak M$.

More generally, there may exist $\mu$ in $K$, a proper subfield $L$ of $K$ and a submodule $\MM '$ of $\MM$ such that $\mu \MM '$ is full-rank in $L$. Again, in this case, our norm-form equation may have infinitely many solutions, for the same reason as above. We say that an $\MM '$ with this property is \emph{proportional} to a full module in $L$. If $\MM$ contains such a submodule $\MM '$ for some subfield $L$ of $K$, and $L$ is not $\QQ$ or imaginary quadratic, then we say that $\MM$ is a degenerate module. A norm-form equation is called degenerate if the norm-form in question spans a degenerate module in $K$. Schmidt's main result in \cite{schmidt1971} implies that non-degenerate norm form equations have only finitely many solutions. In \cite{schmidt1972}, Schmidt essentially shows that once we have accounted for the families discussed above (arising from full modules in subfields), the remaining solutions are finite in number.

In our effective proof for Theorem \ref{sexticthm}, these potential infinite families arise from vanishing subsums in the unit equation. Our first step was to consider $\lVert a_1u_1 + a_2u_2 \rVert_\nu$ and show that this expression is ``large'' if it is non-zero. If it is zero, then letting $\sigma \in \mathrm{Gal}(K/\QQ)$ be such that $\sigma(u_1) = u_2$, we obtain the vanishing subsum $a_1u_1 = b_1\sigma(u_1)$. Solutions to this subsum could give rise to infinitely many solutions to our equation. Alternatively if $a_1u_1 + a_2u_2 \neq 0 $ but two terms can be matched to give zero, then we may again have infinitely many solutions. Since matching is taking place pairwise we again have a subsum involving just two variables, of the form $a_ju_j + a_k\sigma(u_j) = 0$. If $a_1u_1 + a_2u_2 \neq 0$ and after matching we have a three-term equation with non-zero coefficients, then Lemma \ref{threetermlemma} shows we must have finitely many solutions. So if there exist infinitely many solutions to our norm-form equation, then somewhere in our equation we have a subsum of the form $au = b\sigma(u)$ for a unit $u$ and a non-trivial $\sigma \in \mathrm{Gal}(K/\QQ)$.

Suppose there are infinitely many solutions to this subsum, which are also solutions to the norm-form equation. Then fixing an initial solution $u_0$ we see that for any other solution $u$ we must have
\[
u/u_0 = \sigma ( u/u_0)
\]
and so in particular $u/u_0$ always lies in the subfield $K^{\sigma}$ of $K$ fixed by $\sigma$. If there are infinitely many such $u$, then there exists at least one subfield $L$ of $K$ ($L\neq \QQ$ or imaginary quadratic) and a corresponding infinite $L$-family such that the number field $L' = K^\sigma \cap L$ gives rise to infinitely many solutions to our norm-form equation. Now in the case of Theorem \ref{sexticthm} the module $\MM$ has rank at most $5$ and is not full-rank in $K$, so $K^\sigma$ is either quadratic or cubic, and cannot equal $K$. In particular $K^\sigma \cap L = L$, since $K^\sigma \cap L$ gives rise to infinitely many solutions so we cannot have $K^\sigma \cap L = \QQ$. Thus, the entire $L$-family is contained among the solutions to the vanishing subsum, and $(1/u_0)\MM$ contains a submodule that is full-rank in $L$. Thus vanishing subsums in our equation correspond exactly to the families of solutions in the sense of Schmidt \cite{schmidt1972}.

Finally, note that a vanishing subsum of the form 
\[
au = b\sigma(u)
\]
is determined by linear and congruence conditions on exponents of any chosen system of fundamental units. Indeed, let $u_1, u_2$ be a pair of fundamental units, and let $\zeta$ generate the roots of unity in $K$. Then from $au = b\sigma(u)$ we see $b/a$ must be a unit in $K$, say $b/a = \zeta^r u_1^mu_2^n$. Moreover, setting 
\[
u = \zeta^ku_1^{a_1}u_2^{a_2}\,, \quad \sigma(u)  = \zeta^{k'}u_1^{a'_1}u_2^{a'_2}
\]
we can express $k',a'_1,a'_2$ as a linear combination of $k,a_1,a_2$ by checking the action of $\sigma$ on $\zeta,u_1$ and $u_2$. Thus the subsum $u = b\sigma(u)$ is satisfied if and only if 
\begin{align*}
a_1 = a'_1 + m ,\quad a_2 = a'_2+n \quad \text{and} \quad k \equiv k'+r \mod{w}
\end{align*}
where $w$ is the number of roots of unity in $K$. This precisely gives rise to the description for ``effectively determined'' as stated with regard to Theorem \ref{sexticthm}. 

Incidentally, we see that $K$ cannot have a real quadratic subfield -- $K$ would be generated as a cubic extension of this subfield, and every cubic polynomial has a real root, contradicting the assumption that $K$ is totally complex. Thus in fact all infinite families arise from cubic subfields. If $K$ is a CM field, then this cubic subfield is totally real and contains both fundamental units of $K$, so the family of solutions could be quite large-- in particular, one easily finds examples of norm-form equations over CM fields that are satisfied by \emph{every} unit in the field. If $K$ is not CM then any cubic subfield has a unit group of rank one; in this case if we obtain an infinite family of solutions, then this is a one-parameter family (the parameter being the exponent of the fundamental unit of the subfield).

\section{Norm-Form Equations in Higher Degree}

In the case of Theorem \ref{sexticthm}, we saw that matching allowed us to reduce our six-term unit equation to an equation containing only three terms. For Theorem 2, we want to consider how matching can help us solve norm-form equations over number fields of degree $K > 6$. In this case, we have a \emph{system} of unit equations to consider. The matching procedure can still be carried out, and we reduce from $n$ variables to $n/2$, but we need to ensure that $(n/2 - 2)$ equations remain linearly independent after matching. In the special case of equations in a `power basis',  we can guarantee this, and so again we are left with only three-term equations to solve. We now give the proof of this theorem.

\subsection{Generalized Matching}

First, we describe under which situations a condition like \eqref{samesize} holds, so that we can carry out the matching process and reduce the number of terms in our unit equations. Let $\nu$ be any infinite place of $K$ and suppose
\[
\lVert u_1 \rVert_{\nu} \ge \lVert u_2 \rVert_{\nu} \ge \cdots \ge \lVert u_n \rVert_{\nu}
\]
i.e. the $u_j$ are labelled in decreasing order at $\nu$ (if not, just relabel the $u_j$ to ensure this). Additionally, since $K$ is totally complex, we may assume that complex conjugates at $\nu$ are paired together in this ordering, so
\[
\lVert u_{2m-1} \rVert_{\nu} = \lVert u_{2m} \rVert_{\nu} \, , \quad 1 \le m \le n/2.
\]
We claim that if
\begin{equation}\label{m+2}
\lVert u_{n-2} \rVert_{\nu} \ge c_1 \cdot \lVert u_1 \rVert_{\nu}\, h(\overline u)^{d_1},
\end{equation}
then in fact $u_{n-1}$ and $u_n$ can be matched. To see this, we first note that the Galois group of $K$ over $\QQ$ acts transitively on the units $u_j$. The same Galois group acts also on the places of $K$, in a compatible manner -- for example, if $\sigma(u_j) = u_1$ for some $\sigma \in \mathrm{Gal}(K/\QQ)$ and also $\lVert  \cdot \rVert_{\nu'} = \lVert \sigma (\cdot) \rVert_{\nu}$ for some infinite place $\nu'$ of $K$, then $\lVert u_j \rVert_{\nu'} = \lVert u_1 \rVert_{\nu} $ and in particular, $u_j$ is the largest term at $\nu'$. Further, since $n-2$ of the units are ``large'' at $\nu$, we can deduce that $n-2$ of the units are large at any infinite place of $K$ (of course, this need not be the same $n-2$ units for all places). Looking at the Galois action, which is transitive, we see that $u_{n-1}$ and $u_n$ cannot be among the two smallest  terms at any place besides $\nu$, meaning they must be ``large'' (in the sense of \eqref{m+2}) at all places $\nu' \neq \nu$. Immediately we see that $u_{n-1}$ and $u_n$ satisfy \eqref{samesize} for all infinite places of $K$. For $\nu'$ other than $\nu$,  we have
\begin{align*}
 \lVert u_{n-1} \rVert_{\nu'} \ge c \cdot \max_{1\le j\le n} \lVert u_j \rVert_{\nu'} \cdot h(\overline{u})^{-d} \\
 \lVert u_{n} \rVert_{\nu'} \ge c \cdot \max_{1\le j\le n} \lVert u_j \rVert_{\nu'} \cdot h(\overline{u})^{-d}
\end{align*}
for suitable positive constants $c,d$ giving
\[
c\cdot h(\overline{u})^{-d} \le\;\, \lVert u_{n-1}/u_n \rVert_{\nu'} \le c^{-1} \cdot h(\overline{u})^{d}.
\]
At the remaining place $\nu$, the terms $u_n$ and $u_{n-1}$ are complex conjugates so we have a relation like \eqref{samesize} at \emph{all} infinite places of $K$. Thus, by the discussion in Section \ref{matching}, we can match the units $u_{n-1}$ and $u_n$.

As in the sextic case this argument is completely symmetric under the action of $\mathrm{Gal}(K/\QQ)$. In other words, for every infinite place of $K$, we have a distinct pair $u_j$ and $u_k$ that is smallest at the place and large at all the remaining places (assuming of course that $n-2$ terms are large at every place) so we can match $u_j$ with $u_k$. For example, take $\tau \in \mathrm{Gal}(K/\QQ)$, such that $\tau$ is not $1$ and it does not correspond to complex conjugation at any place. Set $u_j = \tau(u_{n-1})$ and $u_k = \tau(u_n)$ and let $\nu'$ be the place of $K$ such that $\lVert \, \cdot \, \rVert_{\nu'} = \lVert \sigma^{-1}(\cdot) \rVert_{\nu}$. Then $u_j$ and $u_k$ are the two smallest terms at $\nu'$, and they must be large at all other places. By exactly the same argument as above, we can match $u_j$ with $u_k$. Thus there is a matched pair for every infinite place of $K$, and every term is part of a matched pair. In effect, we reduce from a system of unit equations in $n$ variables to a system in $n/2$ variables.

\subsection{Proof of Theorem \ref{generalthm}}

Since the case $n=6$ has already been dealt with by Theorem \ref{sexticthm}, we restrict ourselves for the rest of this section to $n \ge 8$. Let us first deal with the case $m=\pm 1$ and $\alpha$ an algebraic integer.  Consider the norm-form equation
\[
N_{K/\QQ}(x + y\alpha + z\alpha^2 + w\alpha^3) = \pm 1,
\]
where $K = \QQ(\alpha)$ is Galois and totally complex, and $\alpha\in \mathcal{O}_K$. Let $\mathrm{Gal}(K/\QQ) = \{\sigma_1,\ldots,\sigma_n\}$ and set $u_j = \sigma_j(x + y\alpha + z\alpha^2 + w\alpha^3)$. The matrix $B$ from Section \ref{prelim} is
\[
B=
\begin{pmatrix}
1& 1 & \cdots & 1 \\
\sigma_1(\alpha) & \sigma_2(\alpha) & \cdots & \sigma_{n}(\alpha) \\ 
\sigma_1(\alpha^2) & \sigma_2(\alpha^2) & \cdots & \sigma_{n}(\alpha^2) \\ 
\sigma_1(\alpha^3) & \sigma_2(\alpha^3) & \cdots & \sigma_{n}(\alpha^3)
\end{pmatrix},
\]
where $\mathrm{Gal}(K/\QQ) = \{\sigma_1,\ldots,\sigma_{n}\}$. Let $1\le i,j,k,\ell\le n$ and define $c_i$ by $\sigma_i(\alpha) = c_i\alpha$ (and similarly for indices $j,k,\ell$). Then the four columns $i,j,k,\ell$ determine the submatrix
\[
\begin{pmatrix}
1& 1 & 1 & 1 \\
c_i\alpha & c_j\alpha & c_k\alpha & c_{\ell}\alpha \\ 
(c_i\alpha)^2 & (c_j\alpha)^2 & (c_k\alpha)^2 & (c_{\ell}\alpha)^2 \\ 
(c_i\alpha)^3 & (c_j\alpha)^3 & (c_k\alpha)^3 & (c_{\ell}\alpha)^3.
\end{pmatrix}
\]
which is a Vandermonde matrix. Since $\alpha$ is primitive, any pair of columns is distinct and thus this submatrix is invertible. We deduce that any four columns of $B$ have rank four. Let $A$ be our matrix of unit equations as in Section \ref{prelim}. Choose any $n-4$ columns of $A$, then permute the columns of $A$ so the chosen $n-4$ columns are the leftmost ones and apply the same permutation to $B$. Now Lemma \ref{lemvoj} implies that our chosen columns of $A$ have rank $n-4$, since any four columns of $B$ have rank $4$. So in fact any $n-4$ columns of $A$ have rank $n-4$. This rank condition will allow us to retain sufficiently many independent equations after matching. We carry out the matching procedure.

Fix an infinite place $\nu$ of $K$. After relabelling $\sigma_1,\ldots,\sigma_n$ we may assume that $\lVert u_1\rVert_\nu \ge \cdots \ge \lVert u_n\rVert_\nu$. Moreover, we may pair the units by complex conjugation at $\nu$, so we have $({u_{2r-1}},{u_{2r}})$ are complex conjugate pairs at $\nu$ for $1\le r\le \tfrac n2$. We can replace $A$ with its reduced row echelon form, i.e.
\[
A = \begin{pmatrix} 
a_1 &  \cdots   & 0 & b_1 & b'_1 &  c_1 & c'_1  \\
\vdots & \ddots    & \vdots & \vdots & \vdots & \vdots & \vdots \\
0 &  \cdots & a_{n-4} & b_{n-4} &  b'_{n-4} &  c_{n-4} &  c'_{n-4}
\end{pmatrix},
\]
where the leftmost $n-4$ columns give a diagonal submatrix since any $n-4$ columns of $A$ were shown to have full rank.  Further, the entries in the rightmost four columns are all non-zero, again since any $n-4$ columns of $A$ have rank $n-4$.

In particular, since $b_1 \neq 0$, the first row tells us
\[
| b_1u_{n-3}|_\nu \gg | a_1u_1|_\nu \quad \text{i.e.} \quad | u_{n-3}|_\nu \gg | u_1 |_\nu
\]
and then from complex conjugation at $\nu$ that $ | u_{n-2}|_\nu \gg | u_1|_\nu$. Thus we see that $n-2$ of the units $u_j$ are ``large'' at $\nu$, i.e. that we have satisfied the condition in \eqref{m+2}. We can therefore match each unit with one other, going from a system of equations in $n$ variables to one in $n/2$ variables. In fact in this case our coefficients only increase by a factor of absolutely bounded height, independent of $h(\overline u)$, unlike the sextic case.

We now permute the columns of $A$ (and relabel the $u_i$ accordingly) to ensure that $({u_{n-3}},u_{n-2})$ and $(u_{n-1}, u_{n})$ are both matched pairs, and row reduce again. Once again, since any $n-4$ columns of $A$ have rank $n-4$, we obtain a diagonal submatrix in the first $n-4$ columns. Thus matching just the rightmost two pairs yields
\[
A = \begin{pmatrix} 
a'_1 &  \cdots   & 0 & * & * &  * & *  \\
\vdots & \ddots    & \vdots & \vdots & \vdots & \vdots & \vdots \\
0 &  \cdots & a'_{n-4} & * & * &  * &  *
\end{pmatrix}  \xrightarrow{\text{match}}
\begin{pmatrix} 
a'_1 &  \cdots   & 0 & * & *  \\
\vdots & \ddots    & \vdots & \vdots & \vdots  \\
0 &  \cdots & a'_{n-4} & * &  * 
\end{pmatrix}
\]
which means we have at most three non-zero terms left in each equation. Three term unit equations can all be solved following Lemma \ref{threetermlemma}, so we deduce an effective bound on the heights of all solutions (although the situation here is again easier than Lemma \ref{threetermlemma} since the heights of the coefficients after matching are absolutely bounded).

To go from the case $m=\pm 1$ to arbitrary $m$, note that the collection of matrices $A_\mu$ as described in Section \ref{prelim} are obtained from $A$ by multiplying each column by a different constant. In particular, these matrices $A_\mu$ still obey the condition that any $n-4$ columns have full rank. Thus the above argument works exactly the same, guaranteeing that we reduce to three-term equations after matching. This completes the proof of Theorem \ref{generalthm}.

\textbf{Remark.} Restricting Theorem \ref{generalthm} to the case of equations in a ``power basis'' allowed us to show that the matrix $B$ satisfied a certain rank condition -- namely that any $4$ columns had full rank. In particular, this means our matching procedure allows us to solve \emph{any} norm-form equation where $K$ satisfies the conditions of Theorem \ref{generalthm} and the associated matrix $B$ satisfies this rank condition.

We could also try to solve equations in five variables over $K$ of higher degree. So long as the equations do not collapse too much after matching, i.e. as long as at least $(n/2-2)$ independent equations are maintained, the proof of Theorem \ref{generalthm} can be carried out to determine the full set of solutions. Again, one may look to the LMFDB for test cases-- the simplest example being for CM fields of degree $8$.  Fields in the database are presented along with a primitive element $\alpha$, and we may consider the equation
\[
N_{K/\QQ}(x+y\alpha+z\alpha^2+w\alpha^3+v\alpha^4) = \pm 1
\]
with this choice of $\alpha$ in each case. For 23229 of the total 29951 octic CM fields available, this equation is solvable by our methods. These examples, among others, will be discussed more thoroughly in \cite{comps}.

\section{Acknowledgements} The author would like to thank Mike Bennett for several discussions,  K{\'a}lm{\'a}n Gy{\H{o}}ry for helpful comments on a preprint version of this article, and the anonymous referees for their suggestions and careful reading.

\bibliographystyle{acm}
\bibliography{Bajpai-NormForm-short}

\end{document}